\documentclass{gtmon_a}
\pdfoutput=1
\usepackage{pinlabel}

\proceedingstitle{Heegaard splittings of 3--manifolds (Haifa 2005)}
\conferencestart{10 July 2005}
\conferenceend{19 July 2005}
\conferencename{Heegaard splittings of 3--manifolds}
\conferencelocation{Haifa}

\editor{Cameron Gordon}
\givenname{Cameron}
\surname{Gordon}

\editor{Yoav}
\givenname{Yoav}
\surname{Moriah}

\title{On the degeneration ratio of tunnel numbers and free tangle
decompositions of knots}

\author{Kanji Morimoto}
\givenname{Kanji}
\surname{Morimoto}
\address{Department of IS and Mathematics\\Konan University\\
Higashi-Nada\\\newline
Okamoto 8-9-1\\Kobe 658-8501\\Japan}
\email{morimoto@konan-u.ac.jp}
\urladdr{}

\volumenumber{12}
\issuenumber{}
\publicationyear{2007}
\papernumber{09}
\startpage{265}
\endpage{275}

\doi{}
\MR{}
\Zbl{}

\arxivreference{} %%% not in arxiv

\received{5 December 2005}
\revised{14 July 2006}
\accepted{16 July 2006}
\published{3 December 2007}
\publishedonline{3 December 2007}
\proposed{}
\seconded{}
\corresponding{}
\version{}

\makeatletter
\def\cnewtheorem#1[#2]#3{\newtheorem{#1}{#3}[section]
\expandafter\let\csname c@#1\endcsname\c@thm}

\let\xysavmatrix\xymatrix
\def\xymatrix{\disablesubscriptcorrection\xysavmatrix}
\AtBeginDocument{\let\bar\wbar\let\tilde\wtilde\let\hat\what}

\theoremstyle{plain}
\newtheorem{thm}{Theorem}[section]
\cnewtheorem{prop}[thm]{Proposition}
\cnewtheorem{cor}[thm]{Corollary}
\theoremstyle{definition}
\cnewtheorem{defn}[thm]{Definition}
\cnewtheorem{prob}[thm]{Problem}
\newtheorem{remark}{Remark}

\makeatother  %  move after \newtheorem block

\makeop{cl}

\begin{document}

\begin{asciiabstract}
In this paper, we introduce a notion called n/k-free
tangle and study the degeneration ratio of tunnel numbers of knots.
\end{asciiabstract}

\begin{htmlabstract}
In this paper, we introduce a notion called n/k&ndash;free
tangle and study the degeneration ratio of tunnel numbers of knots.
\end{htmlabstract}

\begin{abstract}
In this paper, we introduce a notion called $n/k$--free
tangle and study the degeneration ratio of tunnel numbers of knots.
\end{abstract}

\maketitle

\section{Introduction}

Let $K$ be a knot in the 3--sphere $S^3$, $t(K)$ the tunnel number of $K$
and $K_1 \# K_2$ the connected sum of two knots $K_1$ and $K_2$, where
$t(K)$ is the minimal genus $-1$ among all Heegaard splittings which
contain $K$ as a core of a handle. Concerning the relationship between
$t(K_1) + t(K_2)$ and $t(K_1 \# K_2)$, we showed in Morimoto \cite{3} that there
are infinitely many tunnel number two knots $K$ such that $t(K \# K')$ is
two again for any 2--bridge knots $K'$. These are the first examples whose
tunnel numbers go down under connected sum, ie, ``2+1 = 2".
Subsequently, Kobayashi showed in Kobayashi \cite{1}, by taking connected sum of
those knots, that there are infinitely many pairs of knots $(K_1, K_2)$
such that $t(K_1 \# K_2) < t(K_1) + t(K_2) - n$ for any integer $n > 0$.
This shows that tunnel numbers of knots have arbitrarily high degeneration.

Contrary to these phenomena, Scharlemann and Schultens introduced in \cite{6} a
notion called {\em degeneration ratio\/} which is a ratio of $t(K_1 \# K_2)$
and $t(K_1) + t(K_2)$, and showed in \cite{6} that $\displaystyle {t(K_1 \# K_2)
\over {t(K_1) + t(K_2)}} \ge {2 \over 5}$ for any prime knots $K_1$ and
$K_2$. We note that Scharlemann and Schultens's original degeneration ratio
is $\displaystyle 1 - {t(K_1 \# K_2) \over {t(K_1) + t(K_2)}}$, but we use
the above one for convenience.

The degeneration ratio of our first example in Morimoto \cite{3} is
$\displaystyle {2 \over 3}$ because $t(K_1) = 2$, $t(K_2) = 1$ and $t(K_1
\# K_2) = 2$. In fact, this is the smallest example among all we know so
far. In this article, we introduce a notion called $n/k$--free tangle and
study the existence of a pair $(K_1, K_2)$ such that $\displaystyle {t(K_1
\# K_2) \over {t(K_1) + t(K_2)}} < {2 \over 3}$.

Throughout the present paper, we will work in the piecewise linear category.
For a manifold $X$ and a subcomplex $Y$ in $X$, we denote a regular
neighborhood of $Y$ in $X$ by $N(Y, X)$ or simply $N(Y)$.

\section{Free tangles}

Let $M$ be a compact 3--manifold with boundary, and $T = t_1 \cup t_2 \cup
\cdots \cup t_n$ the mutually disjoint arcs properly embedded in $M$. Then
we say that $T$ is a {\em trivial arc system\/} if there are mutually
disjoint disks $\Delta_1, \Delta_2, \ldots , \Delta_n$ in $M$ such that
$\partial \Delta_i = t_i \cup t_i'$ $(i=1,2, \ldots ,n)$, where $t_i'$ is
an arc in $\partial M$.

Let $M = B$ be a 3--ball, then the pair $(B, T)$ is called an $n$--{\em
string tangle\/}. We say that $(B, T)$ is {\em trivial\/} if $T$ is a trivial
arc system in $B$. We say that $(B, T)$ is {\em essential\/} if 
$\cl(\partial B - N(T))$ is incompressible in 
$\cl(B - N(T))$ in the case when $n > 1$ or
$(B, T)$ is not trivial in the case when $n = 1$, where $N(T)$ is a regular
neighborhood of $T$ in $B$. We also say that $(B, T)$ is {\em free\/} if
$\cl(B - N(T))$ is a handlebody.

\begin{defn}[C--trivialization arc system] \label{defn:2.1} Let $(B, T)$ be an
$n$--string tangle, and let $T'$ be a subfamily of $T$. Then we say that
$T'$ is a {\em C--trivialization arc system\/} if $T - T'$ is a trivial arc
system in the 3--manifold $\cl(B - N(T'))$.
\end{defn}

\begin{defn}[$n/k$--free tangle] \label{defn:2.2} 
Suppose $(B, T)$ is an
$n$--string free tangle, and let $k$ be an integer with $0 \le k \le n$.
Then we say that $(B, T)$ is a $n/k$--{\em free tangle\/} if the following
conditions hold:
\begin{enumerate}
\item there is a subfamily $T' \subset T$ with $\#(T') = k$ such that $T'$ is
a C--trivialization arc system,
\item $T''$ is not a C--trivialization arc system for any subfamily $T''
\subset T$ with $\#(T'') < k$.
\end{enumerate}
\end{defn}

\begin{remark} (1) $n/0$--free tangle is a trivial tangle. (2) We say that
$n/n$--free tangle is a {\em full\/} free tangle. Examples of a $2/0$--free
tangle, a $2/1$--free tangle and a $2/2$--free tangle are illustrated in
\fullref{fig:1}. (3) If $T'$ is a C--trivialization arc system in an $n$--string
free tangle $(B, T)$, then $\cl(B - N(T'))$ is a handlebody. Because $T -
T'$ is a trivial arc system in $\cl(B - N(T'))$ and $\cl(B - N(T') - N(T -
T')) = \cl(B - N(T))$ is a handlebody.
\end{remark}

\begin{figure}[ht!]
\begin{center}
\labellist
\small
\pinlabel {C--trivialization arc} at 160 190
\pinlabel {2/0} at 55 40 
\pinlabel {(i)} at 55 10
\pinlabel {2/1} at 210 40
\pinlabel {(ii)} at 210 10 
\pinlabel {2/2 (full)} at 360 40
\pinlabel {(iii)} at 360 10
\endlabellist
\includegraphics[width=10cm]{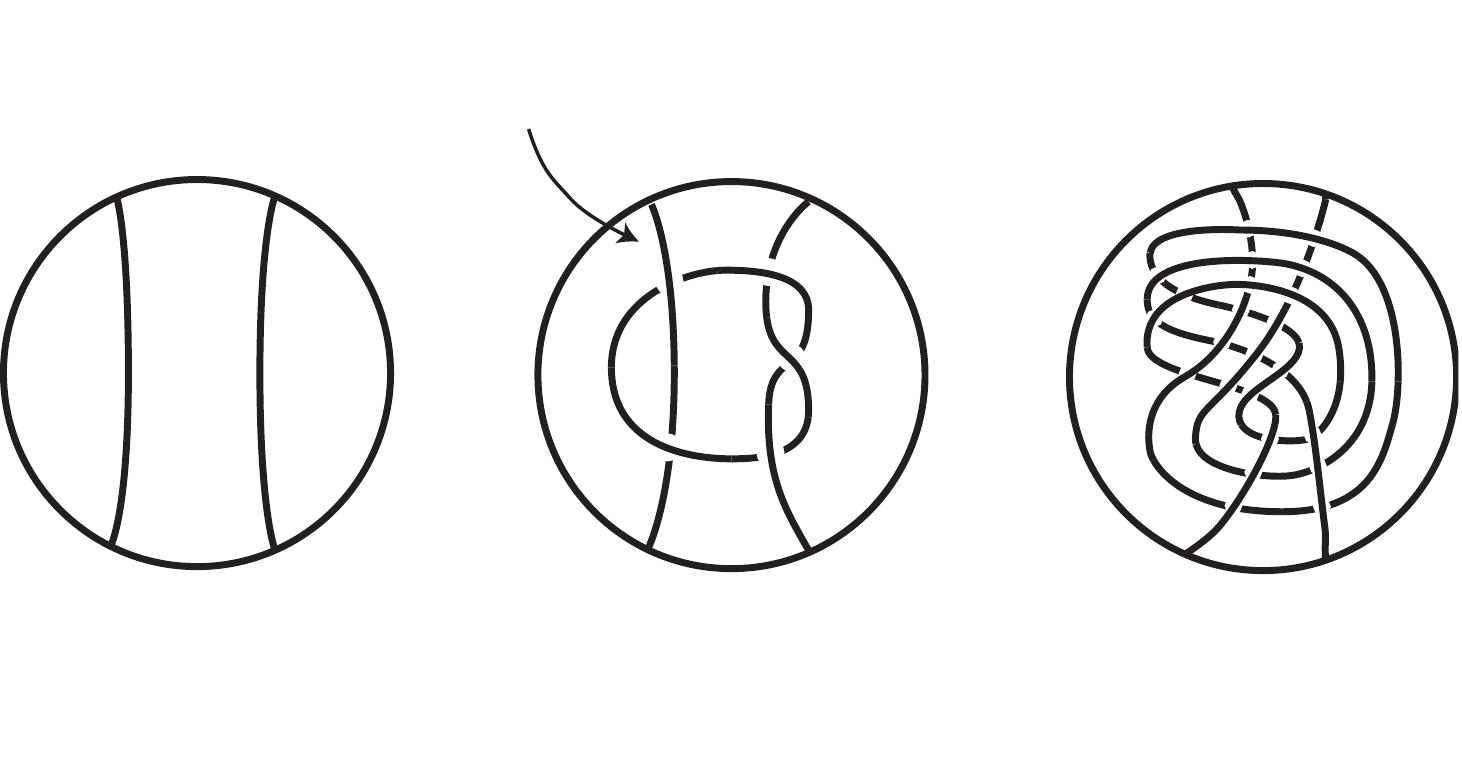}
\caption{}
\label{fig:1}
\end{center}
\end{figure}

We say that a knot $K$ has an $n$--{\em string free tangle decomposition\/} if
$(S^3, K)$ is decomposed into two $n$--string free tangles $(B_1, T_1) \cup
(B_2, T_2)$.

\begin{prop} \label{prop:2.3} Let $K$ be a knot in $S^3$ which has an
$n$--string free tangle decomposition $(S^3, K) = (B_1, T_1) \cup (B_2,
T_2)$. Suppose at least one of $(B_1, T_1)$ and $(B_2, T_2)$ is an
$n/k$--free tangle for some $k$ with $0 \le k \le n$, then 
$t(K) \le n + k - 1$. 
\end{prop}

\begin{proof}
We may assume that $(B_1, T_1)$ is an $n/k$--free tangle, and
put $T_1 = t_1^1 \cup t_2^1 \cup \cdots \cup t_n^1$. Then we can put $T_1'
= t_1^1 \cup \cdots \cup t_k^1$ to be a C--trivialization arc system, and
$T_1' = \emptyset$ if $k = 0$. Let $\alpha_1, \ldots, \alpha_{n-1},
\beta_1, \ldots, \beta_k$ be the arcs in $\partial B_1$ as in 
\fullref{fig:2} so
that $\alpha_i$ connects a point of $\partial t_i^1$ and a point of
$\partial t_{i+1}^1$ $(i=1,2, \ldots, n-1)$, $\beta_1$ connects the two
points of $\partial t_1^1$ and $\beta_i$ connects a point of $\partial
t_{i-1}^1$ and a point of $\partial t_i^1$ $(i=2, \ldots, k)$.

\begin{figure}[ht!]
\begin{center}
\labellist
\small
\pinlabel $\alpha_1$ at 70 240
\pinlabel $\alpha_2$ at 120 250 
\pinlabel $\alpha_{n-1}$ at 380 230 
\pinlabel $t_n^1$ at 435 105
\pinlabel $(B_1,T_1)$ at 350 30
\pinlabel $t_k^1$ at 260 25
\pinlabel $\beta_k$ at 180 15
\pinlabel $\beta_2$ at 80 35
\pinlabel $\beta_1$ at 5 95
\pinlabel $t_1^1$ at 15 190
\endlabellist
\includegraphics[width=8cm]{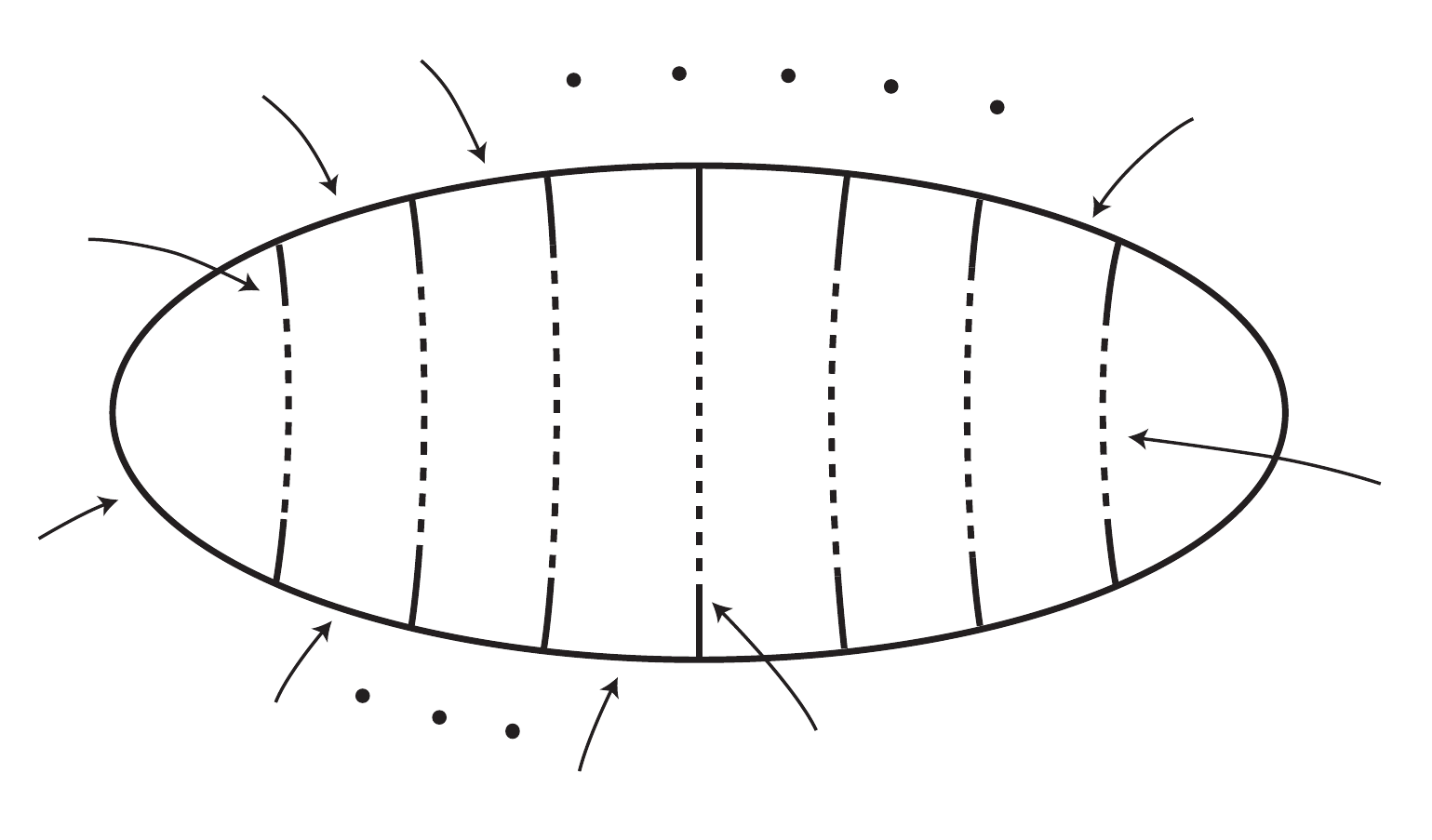}
\caption{}
\label{fig:2}
\end{center}
\end{figure}

Put $D = \cl(\partial B_1 - N(\alpha_1 \cup \cdots \cup \alpha_{n-1} \cup
\beta_1 \cup \cdots \cup \beta_k))$, where $N(\alpha_1 \cup \cdots \cup
\alpha_{n-1} \cup \beta_1 \cup \cdots \cup \beta_k)$ is a regular
neighborhood of $\alpha_1 \cup \cdots \cup \alpha_{n-1} \cup \beta_1 \cup
\cdots \cup \beta_k$ in $S^3$. Then $D$ is a disk in $\partial B_1$ and $D$
is a disk in $\partial B_2$ too. We note that $T_1 - T_1' = t_{k+1}^1 \cup
\cdots \cup t_n^1$ is a trivial arc system in the genus $k$ handlebody
$\cl(B_1 - N(T_1'))$, and one end point of $\partial t_i^1$ 
$(i=k+1, \ldots, n)$ is in $D$.

Put $W_1 = N(K) \cup N(\alpha_1 \cup \cdots \cup \alpha_{n-1} \cup \beta_1
\cup \cdots \cup \beta_k)$, then $W_1$ is a genus $n+k$ handlebody. Put
$W_2 = \cl(S^3 - W_1)$ and put $D' = \cl(D - N(t_{k+1}^1 \cup \cdots \cup
t_n^1))$. Then, by the above note, $W_2 = \cl(B_1 - N(T_1) - N(\alpha_1 \cup
\cdots \cup \alpha_{n-1} \cup \beta_1 \cup \cdots \cup \beta_k)) \cup _{D'}
\cl(B_2 - N(T_2) - N(\alpha_1 \cup \cdots \cup \alpha_{n-1} \cup \beta_1
\cup \cdots \cup \beta_k))$ is a genus $n+n-(n-k) = n+k$ handlebody. Hence
$(W_1, W_2)$ is a genus $n+k$ Heegaard splitting of $S^3$ such that $W_1$
contains $K$ as a core of a handle. This shows that $t(K) \le n + k -1$.
\end{proof}

\begin{cor} \rm{(Morimoto \cite{5})} \qua  \label{cor:2.4} 
If $K$ has an $n$--string free tangle decomposition, then 
$t(K) \le 2n - 1$. 
\end{cor}

By the above proposition, we can ask if the estimate in the proposition is
best possible.

\begin{prob} \label{prob:2.5} For any $n > 1$ and $k$ with $0 \le k \le n$, are
there knots $K$ satisfying the following conditions:
\begin{enumerate}
\item $K$ has an $n$--string free tangle decomposition with at least one
$n/k$--free tangle,
\item $t(K) = n + k - 1$?
\end{enumerate}
\end{prob}

In particular, we want to ask the following.

\begin{prob} \label{prob:2.6} For any $n > 1$, are there knots $K$ satisfying the
following conditions:
\begin{enumerate}
\item $K$ has an $n$--string free tangle decomposition,
\item $t(K) = 2n - 1$?
\end{enumerate}
\end{prob}

\section{Degeneration ratio}

\begin{prop} \label{prop:3.1} 
Let $K_1$ be a knot which has an $n$--string
free tangle decomposition for $n > 1$, and $K_2$ a knot which has an
$(n+1)/0$--free tangle decomposition (ie $n+1$--bridge decomposition). 
Then $t(K_1 \# K_2) \le 2n - 1$. 
\end{prop}

\begin{proof} Suppose $(S_1^3, K_1) = (B_1, T_1) \cup(B_2, T_2)$ is an
$n$--string free tangle decomposition and $(S_2^3, K_2) = (C_1, S_1) \cup
(C_2, S_2)$ is an $(n+1)/0$--free tangle decomposition, where $T_1 = t_1^1
\cup t_2^1 \cup \cdots \cup t_n^1$, $T_2 = t_1^2 \cup t_2^2 \cup \cdots
\cup t_n^2$, $S_1 = s_1^1 \cup s_2^1 \cup \cdots \cup s_{n+1}^1$ and $S_2 =
s_1^1 \cup s_2^2 \cup \cdots \cup s_{n+1}^2$.
Let $N_i^j = N(t_i^j)$ be a regular neighborhood of $t_i^j$ in $B_i$ such
that $N(K_1) = N_1^1 \cup N_2^1 \cup \cdots \cup N_n^1 \cup N_1^2 \cup
N_2^2 \cup \cdots \cup N_n^2$ is a regular neighborhood of $K_1$ in
$S_1^3$, and let $M_i^j = N(s_i^j)$ be a regular neighborhood of $s_i^j$ in
$C_i$ such that $N(K_2) = M_1^1 \cup M_2^1 \cup \cdots \cup M_{n+1}^1 \cup
M_1^2 \cup M_2^2 \cup \cdots \cup M_{n+1}^2$ is a regular neighborhood of
$K_2$ in $S_2^3$.

Divide $t_n^2$ into three arcs $t_{n0}^2 \cup t_{n1}^2 \cup t_{n2}^2$ such
that $t_{n0}^2 \cap t_{n2}^2 = \emptyset$, and divide $N_n^2$ into three
pieces $N_{n0}^2 \cup N_{n1}^2 \cup N_{n2}^2$ according as $t_{n0}^2 \cup
t_{n1}^2 \cup t_{n2}^2$.
Put $N = N_1^1 \cup N_2^1 \cup \cdots \cup N_n^1 \cup N_1^2 \cup N_2^2 \cup
\cdots \cup N_{n-1}^2 \cup N_{n0}^2 \cup N_{n2}^2$, and put $M = M_1^1 \cup
M_2^1 \cup \cdots \cup M_n^1 \cup M_1^2 \cup M_2^2 \cup \cdots \cup
M_{n+1}^2$, ie $N = \cl(N(K_1) - N_{n1}^2)$ and $M = \cl(N(K_2) -
M_{n+1}^1)$. Note that $N \cap N_{n1}^2$ consists of two 2--disks and $M
\cap M_{n+1}^1$ consists of two 2--disks. Then $N$ is a 3--ball in $S_1^3$
and $(N, N \cap K_1)$ is a 1--string trivial tangle, and $M$ is a 3--ball in
$S_2^3$ and $(M, M \cap K_2)$ is a 1--string trivial tangle. We make a
connected sum of $(S_1^3, K_1)$ and $(S_2^3, K_2)$ as follows.
First, by changing the letters if necessary, we may assume that $t_i^1$
connects $t_i^2$ and $t_{i+1}^2$ $(i=1,2, \ldots, n-1)$ and $t_n^1$
connects $t_n^2$ and $t_1^2$, and that $s_i^1$ connects $s_i^2$ and
$s_{i+1}^2$ $(i=1,2, \ldots, n-1)$,  $s_n^1$ connects $s_{n+1}^2$ and
$s_1^2$ and $s_{n+1}^1$ connects $s_n^2$ and $s_{n+1}^2$. Hence we can
identify $N$ and $M$ by the following map $f \co N \to M$.
\begin{align*}
f(N_i^1) & = M_i^1 \qua (i=1,2,\ldots,n)\\
f(N_i^2) & = M_i^2 \qua (i=1,2,\ldots,n-1) \\
f(N_{n0}^2) & = M_n^2 \\
f(N_{n2}^2) & = M_{n+1}^2.
\end{align*}
Put $g = f \vert_{\partial N} \co \partial N \to \partial M$, then by this
glueing map, we get the connected sum $(S^3, K_1 \# K_2) = \cl(S_1^3 - N)
\cup_g \cl(S_2^3 - M)$, where $K_1 \# K_2 = (N_{n1}^2 \cap K_1) \cup
(M_{n+1}^1 \cap K_2)$ as in \fullref{fig:3} $(n=4)$.

\begin{figure}[ht!]
\begin{center}
\labellist
\small
\pinlabel $B_1'$ at 10 190 
\pinlabel $N_1^1$ at 60 170 
\pinlabel $N_2^1$ at 90 170
\pinlabel $N_3^1$ at 120 170
\pinlabel $N_4^1$ at 150 170
\pinlabel $C_1'$ at 240 190
\pinlabel $M_1^1$ at 280 170
\pinlabel $M_2^1$ at 305 170 
\pinlabel $M_3^1$ at 330 170 
\pinlabel $M_4^1$ at 355 170
\pinlabel $M_5^1$ at 380 170 
\pinlabel $N_{41}^2$ at 470 185
\pinlabel $N_{40}^2$ at 150 130
\pinlabel $B_2'$ at 10 40 
\pinlabel $N_1^2$ at 60 15
\pinlabel $N_2^2$ at 90 15
\pinlabel $N_3^2$ at 120 15
\pinlabel $N_{42}^2$ at 150 15
\pinlabel $C_2'$ at 240 40 
\pinlabel $M_1^2$ at 280 15
\pinlabel $M_2^2$ at 305 15
\pinlabel $M_3^2$ at 330 15
\pinlabel $M_4^2$ at 355 15
\pinlabel $M_5^2$ at 380 15
\endlabellist
\includegraphics[width=11cm]{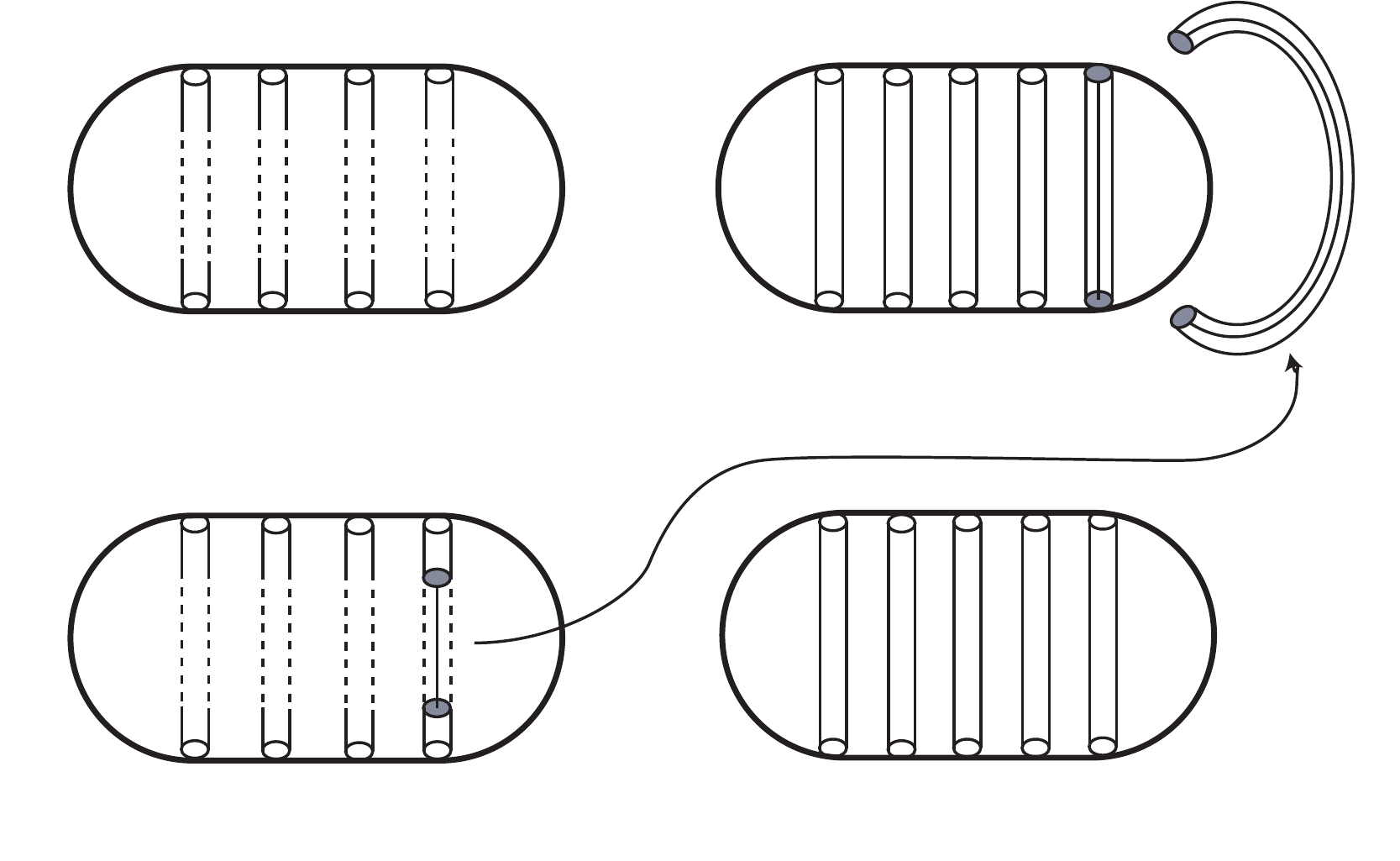}
\caption{}
\label{fig:3}
\end{center}
\end{figure}

Put $B_1' = \cl(B_1 - N)$, $C_1' = \cl(C_1 - M) \cup N_{n1}^2$. Glue
$\partial B_1' \cap \partial N$ and $\partial C_1' \cap \partial M$ with
$g$, and put $W_1 = B_1' \cup_g C_1'$. Then, since $B_1'$ is a genus $n$
handlebody, and since $\{ s_1^1, s_2^1, \ldots, s_n^1 \}$ is a trivial arc
system in $C_1$ and $N_{n1}^2$ is a 1--handle for $C_1$, we see that $W_1$
is a genus $n + (n-1) + 1 = 2n$ handlebody. On the other hand, put $B_2' =
\cl(B_2 - (N \cup N_{n1}^2))$, $C_2' = \cl(C_2 - M)$. Glue $\partial B_2'
\cap \partial N$ and $\partial C_2' \cap \partial M$ with $g$, and put $W_2
= B_2' \cup_g C_2'$. Then, since $B_2'$ is a genus $n$ handlebody, and
since $\{ s_1^2, s_2^2, \ldots, s_{n+1}^2 \}$ is a trivial arc system in
$C_2$, we see that $W_2$ is a genus $n + n = 2n$ handlebody. Hence $(W_1,
W_2)$ is a genus $2n$ Heegaard splitting of $S^3$, and $K_1 \# K_2$ is a
central loop of a handle of $W_1$. This shows that $t(K_1 \# K_2) \le 2n -
1$, and completes the proof of \fullref{prop:3.1}. 
\end{proof}

Suppose there is a knot $K_1$ which has an $n$-string free tangle
decomposition with $t(K_1) = 2n -1$ (cf \fullref{prob:2.6}). Let $K_2$ be a
knot which has an $(n+1)/0$--free tangle decomposition with $t(K_2) = n$ 
(such a knot indeed exists). Then $t(K_1) + t(K_2) = 2n - 1 + n = 3n - 1$,
and by \fullref{prop:3.1}, $t(K_1 \# K_2) \le 2n - 1$. Hence $\displaystyle
{t(K_1 \# K_2) \over {t(K_1) + t(K_2)}} \le {{2n - 1} \over {3n - 1}}$.

In particular, suppose there is a knot $K_1$ which has a 2--string free
tangle decomposition with $t(K_1) = 3$. Then, since there is a knot $K_2$
which has a 3/0--free tangle (3--bridge) decomposition with $t(K_1) = 2$,
we have $t(K_1) = 3$, $t(K_2) = 2$ and $t(K_1 \# K_2) \le 3$ by \fullref{prop:3.1}. 
Moreover, if $t(K_1 \# K_2) = 2$ then $t(K_1) = 1$ or $t(K_2) = 1$ by
[M1, Theorem], a contradiction. Hence $t(K_1 \# K_2) = 3$. This shows that
$\displaystyle {t(K_1 \# K_2) \over {t(K_1) + t(K_2)}} = {3 \over {3 + 2}}
= {3 \over 5} < {2 \over 3}$. Hence, we need to solve the following problem
(a special case of \fullref{prob:2.6}).

\begin{prob} \label{prob:3.2}
Are there (or find) knots $K$ satisfying the
following conditions 
\begin{enumerate}
\item $K$ has a 2--string free tangle decomposition,
\item $t(K) = 3$?
\end{enumerate}
\end{prob}

\begin{remark} \label{rem:2}
If there is a knot $K$ satisfying the conditions in the
above problem, then by \fullref{prop:2.3}, both tangles in the free tangle
decomposition are full free tangles. However, the converse is not true,
because there is a knot $K$ which has a 2--string full free tangle
decomposition but $t(K) = 2$ as follows.

Let $(B_1, T_1)$ be a 2/2--free tangle illustrated in \fullref{fig:1}(iii). Then
$(B_1, T_1)$ is a 2--string full free tangle. Let $(B_2, T_2)$ be a copy of
$(B_1, T_1)$ and put $(S^3, K) = (B_1, T_1) \cup (B_2, T_2)$ with a half
twist. Then, by taking a half twist, $K$ is a knot (not a link) in $S^3$
which has a 2--string full free tangle decomposition. However, by a little
observation, we see that $t(K) = 2$. This shows that the converse is not
true.
\end{remark}

\begin{prop} \label{prop:3.3}
Let $K_1$ be a knot which has an $n$--string
free tangle decomposition with at least one $n/(n-1)$--free tangle for $n >
1$, and $K_2$ a knot which has an $n/0$--free tangle decomposition 
(ie $n$--bridge decomposition). Then $t(K_1 \# K_2) \le 2n - 2$. 
\end{prop}

\begin{proof} Suppose $(S_1^3, K_1) = (B_1, T_1) \cup(B_2, T_2)$ is an
$n$--string free tangle decomposition with an $n/(n-1)$--free tangle, say
$(B_1, T_1)$, and $(S_2^3, K_2) = (C_1, S_1) \cup (C_2, S_2)$ is an
$n/0$--free tangle decomposition, where $T_1 = t_1^1 \cup t_2^1 \cup \cdots
\cup t_n^1$, $T_2 = t_1^2 \cup t_2^2 \cup \cdots \cup t_n^2$, $S_1 = s_1^1
\cup s_2^1 \cup \cdots \cup s_n^1$ and $S_2 = s_1^1 \cup s_2^2 \cup \cdots
\cup s_n^2$.
Let $N_i^j = N(t_i^j)$ be a regular neighborhood of $t_i^j$ in $B_i$ such
that $N(K_1) = N_1^1 \cup N_2^1 \cup \cdots \cup N_n^1 \cup N_1^2 \cup
N_2^2 \cup \cdots \cup N_n^2$ is a regular neighborhood of $K_1$ in
$S_1^3$, and let $M_i^j = N(s_i^j)$ be a regular neighborhood of $s_i^j$ in
$C_i$ such that $N(K_2) = M_1^1 \cup M_2^1 \cup \cdots \cup M_n^1 \cup
M_1^2 \cup M_2^2 \cup \cdots \cup M_n^2$ is a regular neighborhood of $K_2$
in $S_2^3$.

By changing the letters if necessary, we may assume that $t_i^1$ connects
$t_i^2$ and $t_{i+1}^2$ $(i=1,2, \ldots, n-1)$ and $t_n^1$ connects $t_n^2$
and $t_1^2$, and that $s_i^1$ connects $s_i^2$ and $s_{i+1}^2$ $(i=1,2,
\ldots, n-1)$ and $s_n^1$ connects $s_n^2$ and $s_1^2$. Moreover, since
$(B_1, T_1)$ is a $n/(n-1)$--free tangle, we may assume that $t_1^1 \cup
t_2^1 \cup \cdots \cup t_{n-1}^1$ is a C--trivialization arc system in
$B_1$, ie $\cl(B_1 - N(t_1^1 \cup t_2^1 \cup \cdots \cup t_{n-1}^1))$ is
a handlebody and $t_n^1$ is a trivial arc in the handlebody.

Put $N = N_1^1 \cup N_2^1 \cup \cdots \cup N_{n-1}^1 \cup N_1^2 \cup N_2^2
\cup \cdots \cup N_n^2$, and put $M = M_1^1 \cup M_2^1 \cup \cdots \cup
M_{n-1}^1 \cup M_1^2 \cup M_2^2 \cup \cdots \cup M_n^2$, ie $N =
\cl(N(K_1) - N_n^1)$ and $M = \cl(N(K_2) - M_n^1)$. Then $N$ is a 3--ball in
$S_1^3$ and $(N, N \cap K_1)$ is a 1--string trivial tangle, and $M$ is a
3--ball in $S_2^3$ and $(M, M \cap K_2)$ is a 1--string trivial tangle. Hence
we can identify $N$ and $M$ by the following map $f \co N \to M$.
\begin{align*}
f(N_i^1) & = M_i^1 \qua (i=1,2,\ldots,n-1) \\
f(N_i^2) & = M_i^2 \qua (i=1,2,\ldots,n).
\end{align*}

Put $g = f \vert_{\partial N} \co \partial N \to \partial M$, then by this
glueing map, we get the connected sum $(S^3, K_1 \# K_2) = \cl(S_1^3 - N)
\cup_g \cl(S_2^3 - M)$, where $K_1 \# K_2 = (N_n^1 \cap K_1) \cup (M_n^1
\cap K_2)$ as in \fullref{fig:4} $(n=4)$.

\begin{figure}[ht!]
\begin{center}
\labellist
\small
\pinlabel $B_1'$ at 10 195
\pinlabel $N_1^1$ at 65 170 
\pinlabel $N_2^1$ at 95 170 
\pinlabel $N_3^1$ at 125 170
\pinlabel $N_4^1$ at 155 170
\pinlabel $C_1'$ at 230 195 
\pinlabel $M_1^1$ at 290 170
\pinlabel $M_2^1$ at 320 170
\pinlabel $M_3^1$ at 350 170
\pinlabel $M_4^1$ at 380 170
\pinlabel $B_2'$ at 10 45
\pinlabel $N_1^2$ at 65 15
\pinlabel $N_2^2$ at 95 15
\pinlabel $N_3^2$ at 125 15
\pinlabel $N_4^2$ at 155 15
\pinlabel $C_2'$ at 230 45
\pinlabel $M_1^2$ at 290 15
\pinlabel $M_2^2$ at 320 15
\pinlabel $M_3^2$ at 350 15
\pinlabel $M_4^2$ at 380 15
\endlabellist
\includegraphics[width=10cm]{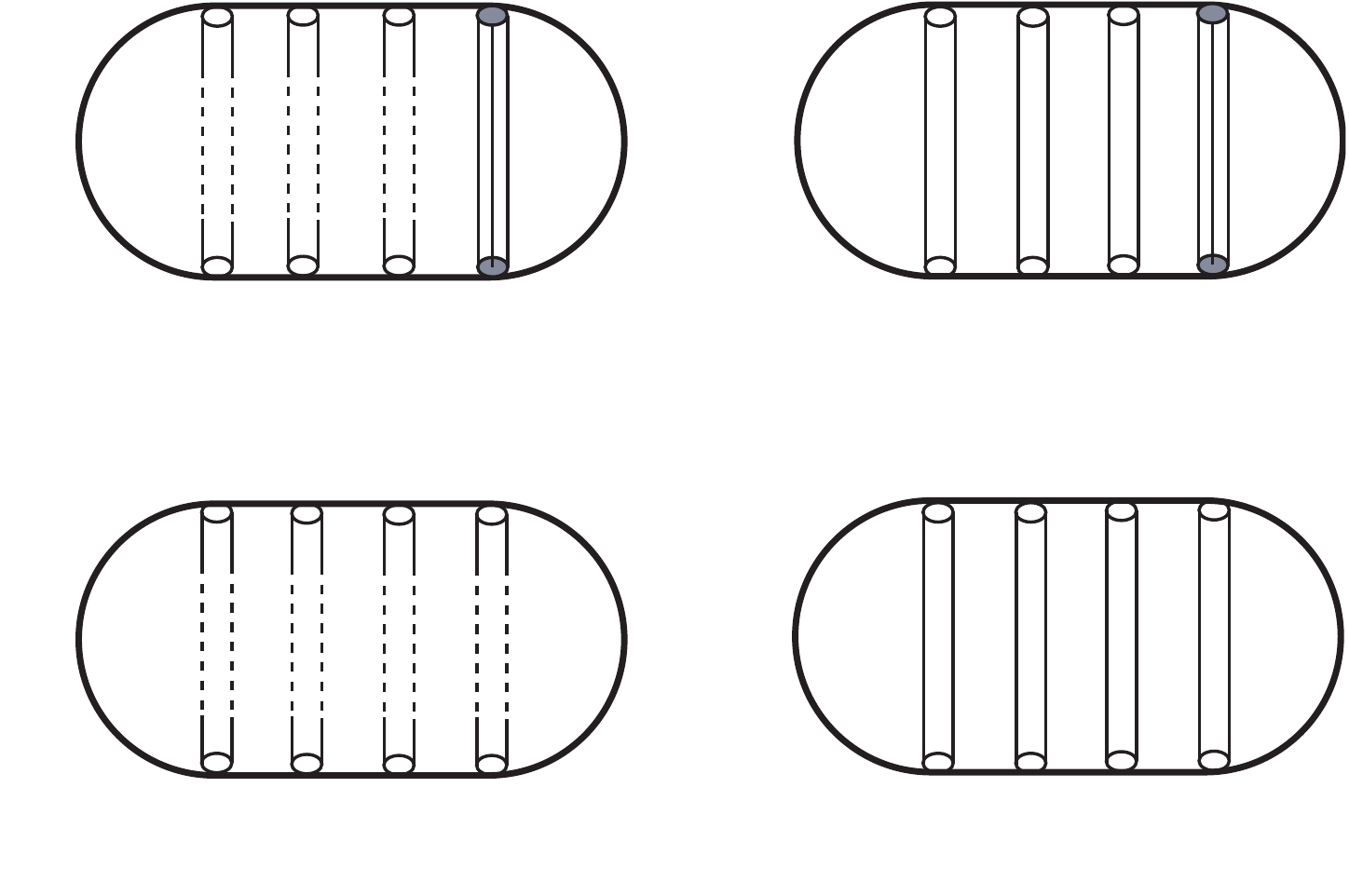}
\caption{}
\label{fig:4}
\end{center}
\end{figure}

Put $B_1' = \cl(B_1 - N)$, $C_1' = \cl(C_1 - M)$. Glue $\partial B_1' \cap
\partial N$ and $\partial C_1' \cap \partial M$ with $g$, and put $W_1 =
B_1' \cup_g C_1'$. Then, since $B_1'$ is a genus $n-1$ handlebody and
$t_n^1$ is a trivial arc in the handlebody, and since $\{ s_1^1, s_2^1,
\ldots, s_n^1 \}$ is a trivial arc system in $C_1$ and $N_n^1 \cap M_n^1$
consists of two 2--disks, we see that $W_1$ is a genus $(n-1) + (n-1) + 1 =
2n-1$ handlebody. On the other hand, put $B_2' = \cl(B_2 - N)$, $C_2' =
\cl(C_2 - M)$. Glue $\partial B_2' \cap \partial N$ and $\partial C_2' \cap
\partial M$ with $g$, and put $W_2 = B_2' \cup_g C_2'$. Then, since $B_2'$
is a genus $n$ handlebody, and since $\{ s_1^2, s_2^2, \ldots, s_n^2 \}$ is
a trivial arc system in $C_2$, we see that $W_2$ is a genus $n + (n-1) =
2n-1$ handlebody. Hence $(W_1, W_2)$ is a genus $2n-1$ Heegaard splitting
of $S^3$, and $K_1 \# K_2$ is a central loop of a handle of $W_1$. This
shows that $t(K_1 \# K_2) \le 2n - 2$, and completes the proof of
\fullref{prop:3.3}. 
\end{proof}

Suppose there is a knot $K_1$ which has an $n$--string free tangle
decomposition with at least one $n/(n-1)$--free tangle and $t(K_1) = 2n-2$
(cf \fullref{prob:2.5}), and let $K_2$ be a knot which has an $n/0$--free tangle
decomposition with $t(K_2) = n-1$ (such a knot indeed exists). Then
$t(K_1) + t(K_2) = (2n-2) + (n-1) = 3n - 3$, and by \fullref{prop:3.3}, $t(K_1
\# K_2) \le 2n - 2$. Hence $\displaystyle {t(K_1 \# K_2) \over {t(K_1) +
t(K_2)}} \le {{2n - 2} \over {3n - 3}} = {{2(n - 1)} \over {3(n - 1)}} ={2
\over 3}$.

In particular, in the case when $n = 2$, there indeed exists a knot $K_1$
which has a 2--string free tangle decomposition with at least one 2/1--free
tangle and $t(K) = 2$ (cf \fullref{fig:1}(ii)), and let $K_2$ be a 2--bridge knot.
Then $t(K_1) = 2$, $t(K_2) = 1$ and $t(K_1 \# K_2) = 2$ by \fullref{prop:3.3}.
Hence $\displaystyle {t(K_1 \# K_2) \over {t(K_1) + t(K_2)}} = {2 \over {2
+ 1}} = {2 \over 3}$. This is the first example whose tunnel numbers go
down under connected sum introduced in Morimoto \cite{3,4}.

In general, for any $n > 1$ and $k$ with $0 \le k \le n$, we have 
the following Theorem.

\begin{thm}\label{thm:3.4} Let $K_1$ be a knot which has an $n$--string
free tangle decomposition with at least one $n/k$--free tangle, and $K_2$ a
knot which has a $(k+1)/0$--free tangle decomposition (ie,
$(k+1)$--bridge decomposition). Then $t(K_1 \# K_2) \le n + k -1$. 
\end{thm}

\begin{proof} If $k = n$ or $n-1$, then this is the same as \fullref{prop:3.1}
or \fullref{prop:3.3} respectively. Hence we assume $k < n-1$.

Suppose $(S_1^3, K_1) = (B_1, T_1) \cup(B_2, T_2)$ is an $n$--string free
tangle decomposition with an $n/k$--free tangle, say $(B_1, T_1)$, and
$(S_2^3, K_2) = (C_1, S_1) \cup (C_2, S_2)$ is an $(k+1)/0$--free tangle
decomposition, where $T_1 = t_1^1 \cup t_2^1 \cup \cdots \cup t_n^1$, $T_2
= t_1^2 \cup t_2^2 \cup \cdots \cup t_n^2$, $S_1 = s_1^1 \cup s_2^1 \cup
\cdots \cup s_{k+1}^1$ and $S_2 = s_1^1 \cup s_2^2 \cup \cdots \cup
s_{k+1}^2$.
Let $N_i^j = N(t_i^j)$ be a regular neighborhood of $t_i^j$ in $B_i$ such
that $N(K_1) = N_1^1 \cup N_2^1 \cup \cdots \cup N_n^1 \cup N_1^2 \cup
N_2^2 \cup \cdots \cup N_n^2$ is a regular neighborhood of $K_1$ in
$S_1^3$, and let $M_i^j = N(s_i^j)$ be a regular neighborhood of $s_i^j$ in
$C_i$ such that $N(K_2) = M_1^1 \cup M_2^1 \cup \cdots \cup M_{k+1}^1 \cup
M_1^2 \cup M_2^2 \cup \cdots \cup M_{k+1}^2$ is a regular neighborhood of
$K_2$ in $S_2^3$.

By changing the letters if necessary, we may assume that $t_i^1$ connects
$t_i^2$ and $t_{i+1}^2$ $(i=1,2, \ldots, n-1)$ and $t_n^1$ connects $t_n^2$
and $t_1^2$, and that $s_i^1$ connects $s_i^2$ and $s_{i+1}^2$ $(i=1,2,
\ldots, k)$ and $s_{k+1}^1$ connects $s_{k+1}^2$ and $s_1^2$. Moreover,
since $(B_1, T_1)$ is a $n/k$--free tangle, we may assume that $t_1^1 \cup
t_2^1 \cup \cdots \cup t_k^1$ is a C--trivialization arc system in $B_1$,
ie, $\cl(B_1 - N(t_1^1 \cup t_2^1 \cup \cdots \cup t_k^1))$ is a
handlebody and $t_{k+1}^1 \cup \cdots \cup t_n^1$ is a trivial arc system
in the handlebody.

Put $N = N_1^1 \cup N_2^1 \cup \cdots \cup N_k^1 \cup N_1^2 \cup N_2^2 \cup
\cdots \cup N_{k+1}^2$, and put $M = M_1^1 \cup M_2^1 \cup \cdots \cup
M_k^1 \cup M_1^2 \cup M_2^2 \cup \cdots \cup M_{k+1}^2$, ie $N =
\cl(N(K_1) - (N_{k+1}^1 \cup \cdots \cup N_n^1 \cup N_{k+2}^2 \cup \cdots
\cup N_n^2))$ and $M = \cl(N(K_2) - M_{k+1}^1)$. Then $N$ is a 3--ball in
$S_1^3$ and $(N, N \cap K_1)$ is a 1--string trivial tangle, and $M$ is a
3-ball in $S_2^3$ and $(M, M \cap K_2)$ is a 1--string trivial tangle. Hence
we can identify $N$ and $M$ by the following map $f \co N \to M$.
\begin{align*}
f(N_i^1) &= M_i^1 \qua (i=1,2, \ldots, k) \\
f(N_i^2) &= M_i^2 \qua (i=1,2, \ldots, k+1).
\end{align*}

Put $g = f \vert_{\partial N} \co \partial N \to \partial M$, then by this
glueing map, we get the connected sum $(S^3, K_1 \# K_2) = \cl(S_1^3 - N)
\cup_g \cl(S_2^3 - M)$, where $K_1 \# K_2 = (((N_{k+1}^1 \cup \cdots \cup
N_n^1) \cup (N_{k+2} \cup \cdots \cup N_n^2)) \cap K_1) \cup (M_{k+1}^1
\cap K_2)$ as in \fullref{fig:5} $(n=6, k=3)$.

\begin{figure}[ht!]
\begin{center}
\labellist
\small
\pinlabel $B_1'$ at 10 190 
\pinlabel $N_1^1$ at 55 170 
\pinlabel $N_2^1$ at 80 170 
\pinlabel $N_3^1$ at 105 170
\pinlabel $N_4^1$ at 130 290
\pinlabel $N_5^1$ at 205 180
\pinlabel $N_6^1$ at 180 170
\pinlabel $B_2'$ at 10 45
\pinlabel $N_1^2$ at 55 15
\pinlabel $N_2^2$ at 80 15
\pinlabel $N_3^2$ at 105 15
\pinlabel $N_4^2$ at 130 15
\pinlabel $N_5^2$ at 155 15
\pinlabel $N_6^2$ at 180 15
\pinlabel $C_1'$ at 265 190
\pinlabel $M_1^1$ at 315 170
\pinlabel $M_2^1$ at 345 170
\pinlabel $M_3^1$ at 375 170
\pinlabel $M_4^1$ at 405 170
\pinlabel $C_2'$ at 265 45
\pinlabel $M_1^2$ at 315 15
\pinlabel $M_2^2$ at 345 15
\pinlabel $M_3^2$ at 375 15
\pinlabel $M_4^2$ at 405 15
\endlabellist
\includegraphics[width=11.5cm]{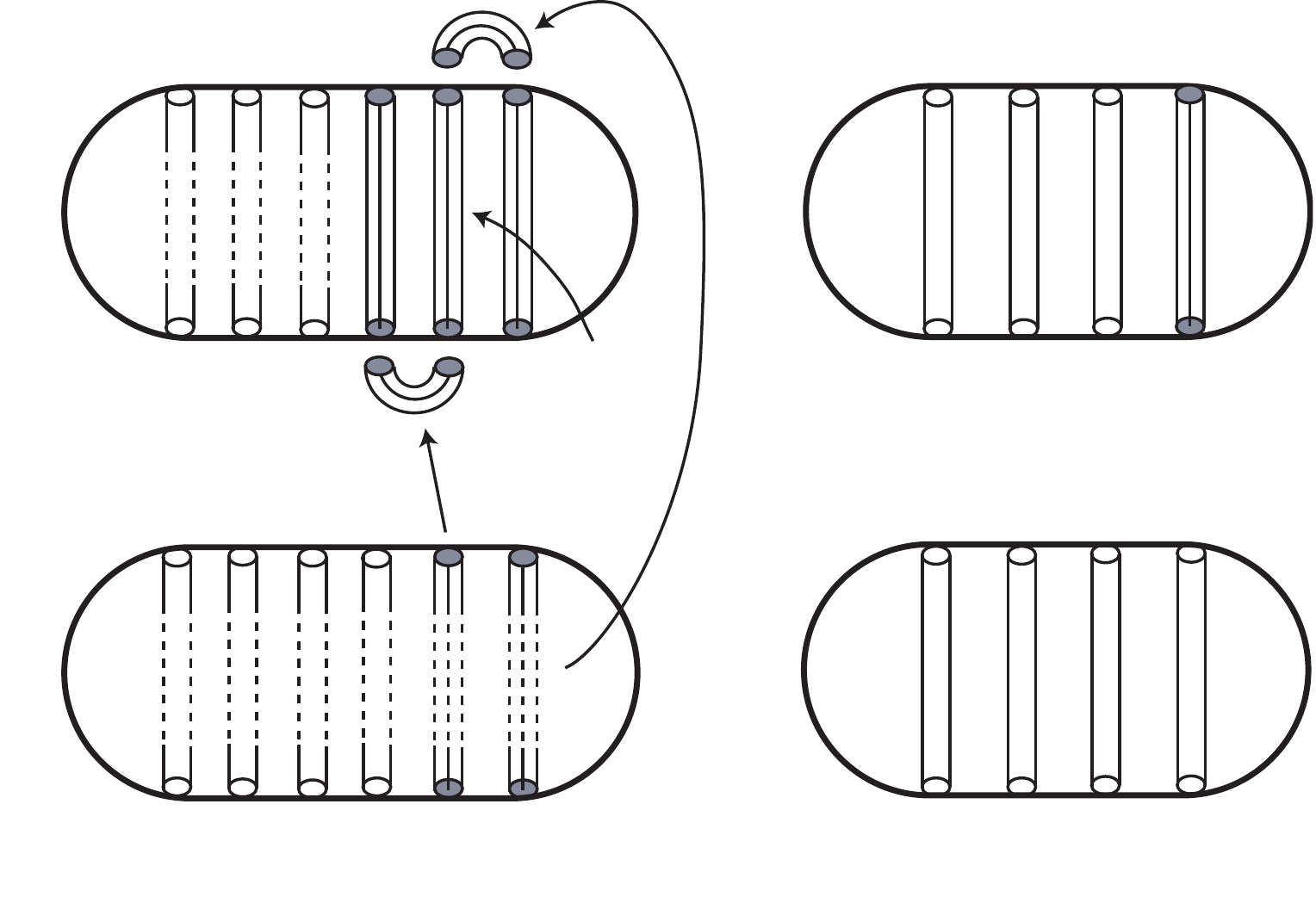}
\caption{}
\label{fig:5}
\end{center}
\end{figure}

Put $B_1' = \cl(B_1 - N) \cup (N_{k+2}^2 \cup \cdots \cup N_n^2)$, $C_1' =
\cl(C_1 - M)$. Glue $\partial B_1' \cap \partial N$ and $\partial C_1' \cap
\partial M$ with $g$, and put $W_1 = B_1' \cup_g C_1'$. Then, since $B_1'$
is a genus $n-1$ handlebody and $t_{k+1}^1 \cup \cdots \cup t_n^1 \cup
t_{k+2}^2 \cup \cdots \cup t_n^2$ is a trivial arc in the handlebody, and
since $\{ s_1^1, s_2^1, \ldots, s_{k+1}^1 \}$ is a trivial arc system in
$C_1$ and $((N_{k+1}^1 \cup \cdots \cup N_n^1) \cup (N_{k+2} \cup \cdots
\cup N_n^2)) \cap M_{k+1}^1$ consists of two 2--disks, we see that $W_1$ is
a genus $(n-1) + (k-1) + 2 = n+k$ handlebody.
On the other hand, put $B_2' = \cl(B_2 - N - N_n^2)$, $C_2' = \cl(C_2 - M)$.
Glue $\partial B_2' \cap \partial N$ and $\partial C_2' \cap \partial M$
with $g$, and put $W_2 = B_2' \cup_g C_2'$. Then, since $B_2'$ is a genus
$n$ handlebody, and since $\{ s_1^2, s_2^2, \ldots, s_{k+1}^2 \}$ is a
trivial arc system in $C_2$, we see that $W_2$ is a genus $n+k$ handlebody.
Hence $(W_1, W_2)$ is a genus $n+k$ Heegaard splitting of $S^3$, and $K_1
\# K_2$ is a central loop of a handle of $W_1$. This shows that $t(K_1 \#
K_2) \le n + k - 1$, and completes the proof of \fullref{thm:3.4}. 
\end{proof}

Suppose there is a knot $K_1$ which has an $n$--string free tangle
decomposition with at least one $n/k$--free tangle and $t(K_1) = n + k - 1$
(cf \fullref{prob:2.5}), and let $K_2$ be a knot which has a $(k+1)/0$--free
tangle decomposition with $t(K_2) = k$  (such a knot indeed exists). Then
$t(K_1) + t(K_2) = n + 2k - 1$, and by \fullref{thm:3.4}, $t(K_1 \# K_2) \le n +
k - 1$. Hence $\displaystyle {t(K_1 \# K_2) \over {t(K_1) + t(K_2)}} \le
{{n + k - 1} \over {n + 2k - 1}}$.

Put $\ell = n - k$, then $0 \le \ell \le n$, $k = n - \ell$, $n + k - 1 =
2n - \ell - 1$ and $n + 2k - 1 = 3n - 2\ell - 1$. Hence $\displaystyle
{t(K_1 \# K_2) \over {t(K_1) + t(K_2)}} \le {{2n - \ell - 1} \over {3n -
2\ell - 1}}$.

If $\ell = 0 \ (k = n)$, then $\displaystyle {{2n - \ell - 1} \over {3n -
2\ell - 1}} = {{2n - 1} \over {3n -  1}} \to {2 \over 3} (-0)$ as $(n \to
\infty)$.

If $\ell = 1 \ (k = n-1)$, then $\displaystyle {{2n - \ell - 1} \over {3n -
2\ell - 1}} = {{2(n-1)} \over {3(n-1)}} = {2 \over 3}$.

If $\ell > 1 \ (k < n-1)$, then $\displaystyle {{2n - \ell - 1} \over {3n -
2\ell - 1}} \to {2 \over 3} (+0)$ as $(n \to \infty)$.

Therefore, we see that the least degeneration ratio can be gotten by the
method in this paper is $\displaystyle {3 \over 5}$ in the case when $n =
2$ and $\ell = 0 \ (k=2)$.

\bibliographystyle{gtart}
\bibliography{link}

\end{document}